\documentclass[a4paper,10pt]{amsart}

\usepackage[mathscr]{eucal}

\usepackage{a4wide}

\usepackage{amsfonts}
\usepackage{amsmath}
\usepackage{amssymb}
\usepackage{amsthm}

\usepackage{graphicx}
\usepackage{float}
\restylefloat{figure}

\newtheorem{theorem}{Theorem}[section]
\newtheorem{lemma}[theorem]{Lemma}

\newtheorem{corollary}[theorem]{Corollary}

\newtheorem{definition}[theorem]{Definition}
\newtheorem{example}[theorem]{Example}
\theoremstyle{remark}

\newtheorem{remark}[theorem]{Remark}

\newcommand{\N}{\mathbb N}
\newcommand{\Z}{\mathbb Z}
\newcommand{\R}{\mathbb R}

\newcommand{\un}{\underline}

\def\l{\lambda_1}

\def\d{\delta}

\begin{document}

\title[]{Necessary and sufficient conditions for periodic decaying resolvents in linear discrete convolution Volterra equations and applications to ARCH$(\infty)$ processes}

\author{John A.~D.~Appleby}
\address{Edgeworth Centre for Financial Mathematics, School of Mathematical Sciences,  \\
Dublin City University, Ireland}
\email{john.appleby@dcu.ie}
\urladdr{http://webpages.dcu.ie/\textasciitilde applebyj}

\author{John A. Daniels}
\address{Edgeworth Centre for Financial Mathematics, School of Mathematical Sciences,  \\
Dublin City University, Ireland}
\email{john.daniels2@mail.dcu.ie}

\thanks{The first author was partially funded by the Science Foundation Ireland grant 07/MI/008
``Edgeworth Centre for Financial Mathematics". The second author was supported by the Irish Research Council for Science, Engineering and Technology under the Embark Initiative grant.}

\subjclass{ 39A06, 39A11, 39A23, 39A30, 39A50, 39A60, 60G10, 62M10, 62P20.}

\keywords{ 
Volterra equation, difference equation, admissibility, 
periodic solution, auto--covariance function, ARCH($\infty$) process.
}

\begin{abstract}
We define a class of functions which have a known decay rate coupled with a periodic fluctuation. We identify conditions on the kernel of a linear summation convolution Volterra equation which give the equivalence of the kernel lying in this class of functions and the solution lying in this class of functions. Some specific examples are examined. In particular this theory is used to provide a counter--example to a result regarding the rate of decay of the auto--covariance function of an ARCH($\infty$) process.
\end{abstract}

\maketitle

\section{Introduction}
This paper characterises the exact decay rate of the solution of the discrete linear Volterra equation
\begin{equation}\label{equ:convol}
	X(n+1) = f(n+1) + \sum_{j=0}^{n}U(n-j)X(j), \quad n\in\Z^+, \quad 	X(0) = X_0,
\end{equation}
where $f:\Z^+\rightarrow\R^d$, $U:\Z^+\rightarrow\R^{d\times d}$ and $X_0\in\R^d$. The exact rate of decay of the forcing function, $f$, is known and the kernel $U$ has known decay and periodic asymptotic behaviour. 
We define the associated resolvent equation of \eqref{equ:convol}
\begin{equation}\label{equ:resres}
	Z(n+1) = \sum_{j=0}^{n}U(n-j)Z(j), \quad n\in\Z^+, \quad Z(0) = I,
\end{equation}
where $Z:\Z^+\rightarrow\R^{d\times d}$ and $I$ is the identity matrix. By first examining \eqref{equ:resres} we can more easily analyse \eqref{equ:convol} via a variation of constants representation:
\begin{equation}\label{equ:resrep}
	X(n) = Z(n)X(0) + \sum_{j=1}^{n}Z(n-j)f(j), \quad n\in\{1,2,...\}.
\end{equation}
	It is shown in \cite{jaigdr:2006} that when the kernel of \eqref{equ:convol} has a particular rate of slower than exponential decay (e.g., polynomial or regularly varying decay), then the solution of \eqref{equ:resres} also has this exact rate of decay. It is from this class of weight function that the rate of decay of $U$ in the present work is imposed.
	It is shown in \cite{songbaker:2003}, \cite{songbaker:2004} and \cite{MR2722631} that periodicity in the kernel of perturbed summation Volterra equations implies periodicity in the solution of these equations. The stability of solutions of perturbed summation Volterra equations is also shown. Linear Volterra convolution and non-convolution equations are studied in \cite{sesm:1998}, where conditions on the summability of the resolvent and stability of the solution are used to establish the existence of a unique bounded (in particular periodic and almost periodic) solution. Conditions guaranteeing  the existence of asymptotically periodic solutions of linear non-convolution summation Volterra equations are derived in \cite{MR2722631} via an application of admissibility theory.
	
	Section \ref{sect:prelim} gives some fundamental definitions as well as various lemmata needed in the proof in Section~\ref{sect:res}.
	 In Section \ref{sect:res} the main result establishes that the solution of \eqref{equ:resres} also decays at the same rate as the kernel and the periodic component is preserved. This result is achieved by eliminating the effect of the periodicity, by evaluating \eqref{equ:resres} at $N$ discrete time points, where $N$ is the value of the period, and lifting the equation to a higher space dimension in which it is asymptotically autonomous. Then	by a careful separation of the summation term we can form a system of equations to which we apply the admissibility theory of~\cite{jaigdr:2006}. Moreover, it can be shown in the case when the kernel is ``small'' in some $\ell^1(\mathbb{Z}^+)$ sense, that $Z$ has periodic 
	 decaying asymptotic behaviour if and only if $U$ does, and indeed both sequences can be majorised by the same weight  function and possess the same period. In forthcoming work, it is planned to investigate more general forms of decay in both continuous and discrete equations, where the decay can be separated into a rate and a bounded component with some structure (such as the periodicity studied here).
Lastly, in Section~\ref{sect:ex} the results developed in Section~\ref{sect:res} are applied to demonstrate that if a periodic fluctuation is present in the kernel of an  ARCH($\infty$) processes then this periodic component propagates through to the auto-covariance function of the  ARCH($\infty$) process. This example sheds further light on extant research on the memory properties of ARCH($\infty$) processes (see e.g.,
\cite{lgpkrl:2000, pkrl:2000, Zaffaroni:2004}).

\section{Preliminary Results}\label{sect:prelim}
If $d$ is a positive integer, the space of all $d\times d$ real matrices is denoted by
$\R^{d\times d}$, the zero matrix by $0$ and the identity matrix by
$I$. A matrix $A=(A_{ij})$ in $\R^{d\times d}$ is \emph{non-negative} if $A_{ij}\geq 0$,
in which case we write $A\geq 0$. A partial ordering is defined on
$\R^{d\times d}$ by letting $A\leq B$ if and only if $B-A\geq 0$. Of
course $A\leq B$ and $C\geq 0$ implies that $CA\leq CB$ and $AC\leq
BC$. The \emph{absolute value} of $A=(A_{ij})$ in $\R^{d\times d}$
is the matrix given by $(|A|)_{ij}=|A_{ij}|$. $\R^{d\times d}$ can
be endowed with many norms, but they are all equivalent. The
\emph{spectral radius} of a matrix $A$ is given by
$\rho(A)=\lim_{n\to\infty}\|A^n\|^{1/n}$, where $\|\cdot\|$ is any
norm on $\R^{d\times d}$; $\rho(A)$ is independent of the norm
employed to calculate it. We note that $\rho(A)\leq
\rho(|A|)$. Also if $0\leq A\leq B$, then $\rho(A)\leq \rho(B)$. Also,
\begin{equation}\label{equ:linalg}
	\rho(A)\leq\ \|A^k\|^{1/k}, \forall k\in\N
\end{equation}
We use, in this paper, the matrix norm $\|A\|_{\infty} = \max\limits_{1\leq i\leq N}\sum_{j=1}^{N}|A_{i,j}|$.

The set of integers is denoted by $\Z$, and $\Z^{+}=\{n\in\Z:n\geq 0\}$.
Sequences $\{u(n)\}_{n\geq 0}$ in $\R^d$ or $\{U(n)\}_{n\geq 0}$ in $\R^{d\times d}$ are sometimes identified
with functions $u:\Z^{+}\to\R^d$ and $U:\Z^{+}\to\R^{d\times d}$. If $\{U(n)\}_{n\geq0}$ and $\{V(n)\}_{n\geq 0}$
are sequences in $\R^{d\times d}$, we define the \emph{convolution}
of $\{(U\ast V)(n)\}_{n\geq0}$ by $(U\ast V)(n)=\sum_{j=0}^n U(n-j)V(j)$ for $n\geq 0$. Moreover using this definition of convolution one may recursively define the \emph{j-fold convolution}, $\{(U^{*j})(n)\}_{j\geq2, n\geq0}$, by 
$(U^{*2})(n) = (U*U)(n)$ and $(U^{*j})(n) = (U^{*(j-1)}*U)(n)$ for $j\geq3$ and $n\geq0$.
The \emph{$Z$-transform} of a sequence $\{U(n)\}_{n\geq0}$ is the sequence in $\R^{d\times d}$ is
defined by $\tilde{U}(z)= \sum_{j=0}^\infty U(j) z^{-j}$,	provided $z$ is a complex number for which the series converges absolutely. A similar definition pertains for sequences with values in other spaces.

Let $C\in\R^{d\times d}$, then we say that $C$ is a circulant matrix if $C_{i,j} = C_{d+i-j+1,1}$ for $i<j$ and $C_{i-j+1,1}$ for $i\geq j$. Such a matrix is a special type of Toeplitz matrix. We introduce a class of weight functions  used throughout this paper, it is defined and studied in \cite{jaigdr:2006}, we state it here for completeness.

\begin{definition}  \label{def:gensubexp}
Let $r>0$. A real-valued sequence
$\gamma=\{\gamma(n)\}_{n\geq0}$ is in $\mathcal{W}(r)$ if
$\gamma(n)>0$ for all $n\geq0$, and
\begin{gather}  \label{eq:p2}
\lim_{n\to\infty} \frac{\gamma(n-1)}{\gamma(n)}=\frac{1}{r},
\quad \tilde{\gamma}(r)=\sum_{i=0}^\infty  \gamma(i)r^{-i}<\infty,\\
\lim_{m\to\infty} \biggl(\limsup_{n\to\infty}\frac{1}{\gamma(n)}\sum_{i=m}^{n-m} \gamma(n-i)\gamma(i)\biggr)=0.
\label{eq:p1}\end{gather}
\end{definition}

Observe that if $r<1$ and $\gamma\in \mathcal{W}(r)$, then $\gamma$ decays; whereas if $r>1$, $\gamma$ diverges.
Criteria for showing that a sequence $\{\gamma(n)\}_{n\geq0}$ is in
$\mathcal{W}(r)$ are given in \cite{jaigdr:2006}. Here we
simply note that $\gamma(n)=r^{n}n^{-\alpha}$ for $\alpha>1$;
$\gamma(n)=r^{n}n^{-\alpha}\exp(-n^{\beta})$ for
$\alpha\in\mathbb{R}$, $0<\beta<1$; and $\gamma(n)=r^{n} e^{-n/(\log
n)}$ all define sequences in $\mathcal{W}(r)$. The sequences defined
by $\gamma(n) =r^{n}$, and $\gamma(n)=r^{n}n^{-\alpha}$, $\alpha\leq 1$ are \emph{not} in $\mathcal{W}(r)$.

In this paper, we investigate a class of kernels which have the essential rate of decay of a sequence in $\mathcal{W}(r)$, but exhibit a periodic ``fluctuation'' of period $N\in\mathbb{N}$ around this rate of decay. To encapsulate this idea we give the following definition.
\begin{definition}
	Let $d,N\in\Z^{+}/\{0\}$ and $r>0$ be finite. A sequence $U=\{U(n)\}_{n\geq0}\in\R^{d\times d}$ is in $\mathcal{WP}(r,N)$ if there exists a function $\phi\in\mathcal{W}(r)$ and a sequence of $d \times d$ matrices $\{A_i\}_{i=0}^{N-1}$ such that
	$\lim_{n\to\infty} U(Nn+i)/\phi(Nn)=A_i$. 	We refer to $\phi$ as a \text{\rm weight function} for $U$.
\end{definition}
If we wish to investigate the rate of decay of a function relative to a particular weight function, say $\gamma$, then it is desirable to know how $\gamma(Nn)$ relates to $\gamma(n)$.
\begin{lemma}\label{lemma:wgt}
		Let $N$ be a positive integer and $r>0$.
		If $\phi\in\mathcal{W}(r)$ then $\Phi\in\mathcal{W}(\tau)$, where $\Phi(n):=\phi(Nn)$ and $\tau:=r^N$
\end{lemma}
\begin{proof}
Note that $\Phi(n)=\phi(Nn)>0$. We establish \eqref{eq:p2} and \eqref{eq:p1} for $\Phi$. Since 
	$\Phi(n-1)/\Phi(n) = \phi(Nn-N)/\phi(Nn)$ 
	and $\phi$ obeys \eqref{eq:p2}, we get	$\lim_{n\to\infty} \Phi(n-1)/\Phi(n) = 1/r^N = 1/\tau$.
	Also
	\begin{gather*}
		\widetilde{\Phi}(\tau) = \sum_{i=0}^{\infty}\Phi(i)\tau^{-i} = \sum_{i=0}^{\infty}\phi(Ni)r^{-Ni} \leq \sum_{i=0}^{\infty}\phi(i)r^{-i} < \infty .
	\end{gather*}
	Turning to \eqref{eq:p1}, by construction we have
	\begin{gather*}
		\sum_{i=m}^{n-m}\frac{\Phi(n-i)\Phi(i)}{\Phi(n)} = \sum_{i=m}^{n-m}\frac{\phi(Nn-Ni)\phi(Ni)}{\phi(Nn)}
		\leq \sum_{i=Nm}^{Nn-Nm}\frac{\phi(Nn-i)\phi(i)}{\phi(Nn)}.
	\end{gather*}
	Therefore
	\begin{align*}
		\limsup_{n\to\infty}\sum_{i=m}^{n-m}\frac{\Phi(n-i)\Phi(i)}{\Phi(n)}
		&\leq \limsup_{n\to\infty}\sum_{i=Nm}^{Nn-Nm}\frac{\phi(Nn-i)\phi(i)}{\phi(Nn)}
		\leq \limsup_{L\to\infty}\sum_{i=Nm}^{L-Nm}\frac{\phi(L-i)\phi(i)}{\phi(L)}.
	\end{align*}
	The last inequality is obtained by letting $L=Nn$ and noting that in the limit the sum to $L-Nm$ will contain more terms than $Nn-Nm$. Finally, as $\phi\in\mathcal{W}(r)$
	\begin{align*}
		\limsup_{m\to\infty}\limsup_{n\to\infty}\sum_{i=m}^{n-m}\frac{\Phi(n-i)\Phi(i)}{\Phi(n)}
		&\leq \limsup_{m\to\infty}\limsup_{L\to\infty}\sum_{i=Nm}^{L-Nm}\frac{\phi(L-i)\phi(i)}{\phi(L)} \\
		&\leq \limsup_{P\to\infty}\limsup_{L\to\infty}\sum_{i=P}^{L-P}\frac{\phi(L-i)\phi(i)}{\phi(L)} = 0,
	\end{align*}
	with the last inequality holding by reasoning similar to above.
\end{proof}

In determining the  results in Section \ref{sect:res} we have used \cite[Thm.3.2]{jaigdr:2006} which we state here for completeness. Note in this result and the rest of the paper that if $\gamma$ is a positive real sequence, $f\in\mathbb{R}^{d_1\times d_2}$, and
$\lim_{n\to\infty} f(n)/\gamma(n)$ exists we denote the limit by $L_\gamma f$. The theorem provides an explicit formula for $L_{\gamma}z$ in terms of the data.
\begin{theorem}\label{thm:3}
Let $f:\Z^+\to \R^d$  and $F:\Z^+\to \R^{d\times d}$ and suppose $\{z(n)\}_{n\geq0}$ obeys
\begin{equation}\label{equ:1a}
	z(n+1) = f(n) + \sum_{i=0}^{n}F(n-i)z(i), \quad n\geq0, \quad	z(0)= z_0 \in \R^{d}.
\end{equation}
Suppose that there is
a $\gamma$ in $\mathcal{W}(r)$ such that $L_{\gamma}f$ and $L_{\gamma}F$ both exist, and that
\begin{equation}\label{eq:thm3a}
    \rho\bigl(r^{-1}\widetilde{|F|}(r)\bigr)=\rho\biggl(\sum_{i=0}^{\infty}r^{-(i+1)}|F(i)|\biggr)<1.
\end{equation}
Then the solution $z$ of \eqref{equ:1a} satisfies
\begin{equation}\label{eq:thm3b}
    L_{\gamma}z=(rI- \tilde{F}(r))^{-1}[L_{\gamma}f+(L_{\gamma}F) \tilde{z}(r)],
\end{equation}
where
  $\tilde{z}(r)=(rI- \tilde{F}(r))^{-1}[rz_0+\tilde{f}(r)]$.
\end{theorem}

We provide a preliminary lemma which demonstrates that the inverse of a lower triangular block Toeplitz matrix is also a lower triangular block Toeplitz matrix.
\begin{lemma}\label{lemma:matinv}
	 Let $B_{2,1}, B_{3,1},...,B_{N,1}$ be $d\times d$ matrices. Let $B$ be a matrix in $\R^{Nd\times Nd}$ with $N,d\in\Z^+$ such that $B$ has the following block structure, for $i,j=\{1,...,N\}$,
	\[
 		B_{i,j} =
 		\begin{cases}
   		0_d, & \text{if } i \leq j, \\
   		B_{i-j+1,1}, & \text{if } i>j,
  	\end{cases}
	\]
	where $0_{d}$ represents the $d\times d$ zero matrix.	
	Then $(I-B)^{-1}$ exists and setting $C:=(I-B)^{-1}$ we have  	
	\begin{equation}\label{equ:b1}
		C_{i,j} =
		\begin{cases}
			0_d, &\text{if } i<j, \\
			I_d, &\text{if } i=j, \\
			C_{i-1,j-1}, &\text{if } i>j>1.
		\end{cases}
	\end{equation}
	and
	\begin{equation}\label{equ:b2}
		C_{t,1} = \sum_{l=1}^{t-1}B_{l+1,1}C_{t-l,1}
		= \sum_{l=1}^{t-1}C_{t-l,1}B_{l+1,1} \quad \text{for } t\geq2.
	\end{equation}
\end{lemma}
\begin{proof}
	Note that $I-B$ has ones on its main diagonal (ie. $\det(I-B)=1\not=0$) and hence is invertible.
	The lower triangular structure of $C$ is determined by considering the $i,j^{th}$ element of $(I-B)C$ and using an induction argument.
	We start  by establishing the relation
	\begin{equation}\label{eq:matinv}
		C_{i,j} = \sum_{l=j}^{i-1}B_{i,l}C_{l,j} = \sum_{l=j+1}^{i}C_{i,l}B_{l,j}, \quad \text{ for } i>j.
	\end{equation}
First, we	observe that
	\[
		0_d = [C(I-B)]_{i,j} = \sum_{l=j}^{i}C_{i,l}(I-B)_{l,j} = C_{i,j} - \sum_{l=j+1}^{i}C_{i,l}B_{l,j}
	\]
	By similarly considering $[(I-B)C]_{i,j}$, one establishes \eqref{eq:matinv}.
	We use induction to establish the third equality of \eqref{equ:b1}, which is equivalent to
	\begin{equation}\label{eq:toep}
		C_{i,j} = C_{i-j+1,1}, \quad \text{ for } i>j.
	\end{equation}
	We first prove $C_{j+1,j} = C_{2,1}$. From \eqref{eq:matinv}
	\[
		C_{j+1,j} = \sum_{l=j}^{j}B_{j+1,l}C_{l,j} = B_{j+1,j}C_{j,j} = B_{2,1}C_{1,1} = \sum_{l=1}^{2-1}B_{2,l}C_{l,1} = C_{2,1}
	\]
	Now, assume $C_{p,q} = C_{p-q+1,1}$ for all $0\leq p-q<i-j$ and $p,q\in\{1,...,N\}$ and $i,j$ are fixed.
	\begin{align*}
		C_{i,j} &= \sum_{l=j}^{i-1}B_{i,l}C_{l,j} =\sum_{l=1}^{i-j}B_{i,l+j-1}C_{l+j-1,j} \\
		&= \sum_{l=1}^{i-j+1-1}B_{i-j+1,l}C_{l+j-1,j} 
		= \sum_{l=1}^{i-j+1-1}B_{i-j+1,l}C_{l,1} = C_{i-j+1,1}.
	\end{align*}
	Thus one has $C_{i,j} = C_{i-j+1,1}$ for all $i>j$.	With \eqref{eq:matinv} and \eqref{eq:toep} established, we can conclude \eqref{equ:b2}.
\end{proof}

We supply a Lemma which will be used in the proof of the main result, Theorem \ref{thm:res}.
\begin{lemma}\label{lm:spec}
	Let $\{U(n)\}_{n\in \Z^{+}}$ be a sequence in $\R^{d \times d}$. Suppose
	\begin{equation}\label{eq:hyp}
	\max_{1\leq p\leq d} \left(\sum_{q=1}^{d} \sum_{i=0}^{N-1}\sum_{l=0}^{\infty}r^{-N(l+1)}|U(Nl+i)|_{p,q} \right) < 1 ,  \quad r\leq1.
	\end{equation}
	Define, for some $N\in\{1,2,\dots\}$, the matrix function $F:\Z^+ \to \R^{N\times N}$ by $F(n) = (I-B)^{-1}J(n)$ for $n\geq 1$,
	where the $d \times d$ block composition of $B$ and $J$, for $i,j\in\{1,..N\}$, is given by
	\begin{equation}\label{eq:iminusb}
 		[I-B]_{i,j} =
 		\begin{cases}
   		0_d, & \text{ if } i < j, \\
   		I_d, & \text{ if } i = j, \\
   		-U(i-j-1), & \text{ if } i>j,
  	\end{cases}
  	\qquad 
		[J(n)]_{i,j} =
		\begin{cases}
			U(Nn+ N+i-j-1), & \text{ if } i\leq j, \\
			U(N(n+1) + i-j-1) & \text{ if } i> j.
		\end{cases}
	\end{equation}
	Then
	\begin{equation}\label{eq:rownorm}
		\left\|\sum_{i=0}^{\infty}r^{-N(i+1)}|F(i)| \, \right\|_{\infty} < 1.
	\end{equation}
\end{lemma}
Although the entries $[J(n)]_{i,j}$ of $J(n)$ have the same form for all $i$ and $j$, it is convenient in the proof to express them in the slightly differing forms displayed above.
\begin{proof}
We use the notation, for $\lambda\in\{0,\cdots N-1\}$,
$S_\lambda := \sum_{l=0}^{\infty}r^{-Nl}|U(Nl+\lambda)|$ , $S := \sum_{l=0}^{N-1}S_l$ and $M := \sum_{i=0}^{\infty}r^{-N(i+1)}|F(i)|$. 	Note by \eqref{eq:hyp} that $\| r^{-N}S\|_{\infty}<1$.
	Hence,
	\[
 		0\leq\left[\sum_{l=0}^{\infty}r^{-N(l+1)}|J(l)|\right]_{i,j} =
  	\begin{cases}
   		r^{-N}S_{N+i-j-1}, & \text{if } i \leq j, \\
   		S_{i-j-1} - |U(i-j-1)|, & \text{if } i > j.
  	\end{cases}
	\]
	Also, for $i> j$ and by noting that $(I-B)$ is a matrix of the form in Lemma \ref{lemma:matinv}, we use \eqref{equ:b1}
	\begin{align*}
		M_{i,j} &\leq \sum_{k=1}^{N}|(I-B)^{-1}|_{i,k}\left[\sum_{n=0}^{\infty}r^{-N(l+1)}|J(l)|\right]_{k,j} \\
		&= \sum_{k=1}^{j}|C|_{i,k}\left[\sum_{n=0}^{\infty}r^{-N(l+1)}|J(l)|\right]_{k,j} +
		 \sum_{k=j+1}^{i}|C|_{i,k}\left[\sum_{n=0}^{\infty}r^{-N(l+1)}|J(l)|\right]_{k,j} \\
		 &= \sum_{k=1}^{j}|C|_{i,k}r^{-N}S_{N+k-j-1}+\sum_{k=j+1}^{i}|C|_{i,k}(S_{k-j-1}-|U(k-j-1)|),
	\end{align*}
	where $C:=(I-B)^{-1}$. Similarly	for $i\leq j$ we have
$M_{i,j} \leq \sum_{k=1}^{i}|C|_{i,k}r^{-N}S_{N+k-j-1}$.
	We note that, by definition $M$ is a non-negative matrix, that is, in verifying \eqref{eq:rownorm} we need consider the row sums of $M$ rather than $|M|$.
	We now compute the sum of each row of $M$ and show that they are all less than one. The sum of the first and second rows are special cases. We compute the sum for the first row and also the general case; the sum for the second row is similar to the general case.

	For $i=1$,
	\[
		\sum_{j=1}^{N}M_{1,j} \leq \sum_{j=1}^{N}|C|_{1,1}r^{-N}S_{N-j} = r^{-N}\sum_{j=1}^{N}S_{N-j} = r^{-N}S.
	\]
	Indeed,
	\[
		\left\|\sum_{j=1}^{N}M_{1,j}\right\|_{\infty} = \max_{1\leq p\leq d}\sum_{q=1}^{d}[\sum_{j=1}^{N}M_{1,j}]_{p,q}
		\leq \max_{1\leq p\leq d}\sum_{q=1}^{d}r^{-N}[S]_{p,q} <1.
	\]
	For $i\geq 3$
	\begin{align*}
		\sum_{j=1}^{N}M_{i,j} &= \sum_{j=1}^{i-1}M_{i,j}+\sum_{j=i}^{N}M_{i,j} \\
		&\leq \sum_{j=1}^{i-1}\sum_{k=1}^{j}|C|_{i,k}r^{-N}S_{N+k-j-1} + \sum_{j=1}^{i-1}\sum_{k=j+1}^{i}|C|_{i,k}S_{k-j-1} \\
		&\qquad - \sum_{j=1}^{i-1}\sum_{j+1}^{i}|C|_{i,k}|U(k-j-1)| + \sum_{j=i}^{N}\sum_{k=1}^{i}|C|_{i,k}r^{-N}S_{N+k-j-1} \\
		&= \sum_{k=1}^{i-1}|C|_{i,k}\sum_{j=k}^{i-1}r^{-N}S_{N+k-j-1} + \sum_{k=2}^{i}|C|_{i,k}\sum_{j=1}^{k-1}S_{k-j-1} \\
		&\qquad - \sum_{k=2}^{i}|C|_{i,k}\sum_{j=1}^{k-1}|U(k-j-1)| + \sum_{k=1}^{i}|C|_{i,k}\sum_{j=i}^{N}r^{-N}S_{N+k-j-1}.
	\end{align*}
	By moving the $k=i$ terms from the second and fourth sum, and combining the first and fourth sum we get
	\begin{gather} 
		\sum_{j=1}^{N}M_{i,j} \leq 
		\sum_{k=1}^{i-1}|C|_{i,k}\sum_{j=k}^{N}r^{-N}S_{N+k-j-1}
		+ \sum_{k=2}^{i-1}|C|_{i,k}\sum_{j=1}^{k-1}S_{k-j-1}
		\notag \\
		- 
		\sum_{k=2}^{i}|C|_{i,k}\sum_{j=1}^{k-1}|U(k-j-1)|
		+ 
		\sum_{j=i}^{N}r^{-N}S_{N-j+i-1} + \sum_{j=1}^{i-1}S_{i-j-1}
		:= A_2-A_3+A_1. \label{eq:Asum}
	\end{gather}
	where the first two sums are $A_2$, the next is $A_3$ and the last two are $A_1$. 
Next, we write $A_1$ as
	\begin{equation}\label{eq:A1}
		A_1 = \sum_{l=i-1}^{N-1}r^{-N}S_l + \sum_{l=0}^{i-2}S_l = r^{-N}S + (1-r^{-N})\sum_{l=0}^{i-2}S_{l}.
	\end{equation}
	As for $A_2$ we rearrange to get
	\begin{align}
		A_2 &= \sum_{k=1}^{i-1}|C|_{i,k}\sum_{l=k-1}^{N-1}r^{-N}S_l + \sum_{k=2}^{i-1}|C|_{i,k}\sum_{l=0}^{k-2}S_l \notag\\
		&= \sum_{k=2}^{i-1}|C|_{i,k}r^{-N}\sum_{l=0}^{N-1}S_l - \sum_{k=2}^{i-1}|C|_{i,k}r^{-N}\sum_{l=0}^{k-2}S_l
		+ |C|_{i,1}\sum_{l=0}^{N-1}r^{-N}S_l + \sum_{k=2}^{i-1}|C|_{i,k}\sum_{l=0}^{k-2}S_l \notag\\
		&= \sum_{k=2}^{i-1}|C|_{i,k}r^{-N}S + (1-r^{-N})\sum_{k=2}^{i-1}|C|_{i,k}\sum_{l=0}^{k-2}S_l + |C|_{i,1}r^{-N}S. \label{eq:A2}
	\end{align}
	Regarding $A_3$, we note that by \eqref{eq:toep} and \eqref{equ:b2}
	$C_{i,k} = C_{i-k+1,1} = -\sum_{l=0}^{i-k-1}C_{i-k-l,1}U(l)$ for $i>k$. 		Therefore
	\begin{align}\label{eq:A3}
		A_3 &= \sum_{l=2}^{i}|C|_{i,l}\sum_{k=1}^{l-1}|U(l-1-k)| = \sum_{k=1}^{i-1}\sum_{l=k+1}^{i}|C|_{i,l}|U(l-k-1)|
		= \sum_{k=1}^{i-1}\sum_{l=1}^{i-k}|C|_{i,l+k}|U(l-1)| \notag \\
		&= \sum_{k=1}^{i-1}\sum_{l=1}^{i-k+1-1}|C|_{i-l-k+1,1}|U(l-1)|
		= \sum_{k=1}^{i-1}\sum_{l=0}^{i-k-1}|C|_{i-l-k,1}|U(l)|
		\geq \sum_{k=1}^{i-1}|C|_{i-k+1,1} = \sum_{k=1}^{i-1}|C|_{i,k}.
	\end{align}
	Inserting \eqref{eq:A1}, \eqref{eq:A2} and \eqref{eq:A3} into \eqref{eq:Asum} we can write.
	\begin{align*}
		\sum_{j=1}^{N}M_{ij} &\leq r^{-N}S + (1-r^{-N})\sum_{l=0}^{i-2}S_{l} + \sum_{k=2}^{i-1}|C|_{ik}r^{-N}S  
		\\ &\qquad 
	   + (1-r^{-N})\sum_{k=2}^{i-1}|C|_{ik}\sum_{l=0}^{k-2}S_l + |C|_{i1}r^{-N}S 
	   -\sum_{k=1}^{i-1}|C|_{ik}  \\
		&= r^{-N}S + (1-r^{-N})\sum_{l=0}^{i-2}S_l + \sum_{k=1}^{i-1}|C|_{i,k}(r^{-N}S-I_d)	
		+ (1-r^{-N})\sum_{k=2}^{i-1}|C|_{i,k}\sum_{l=0}^{k-2}S_l. 
	\end{align*}
	We note that by conditions \eqref{eq:hyp} we have $1-r^{-N}\leq0$. Therefore
	\[
		\sum_{j=1}^{N}M_{i,j} \leq r^{-N}S + \sum_{k=1}^{i-1}|C|_{i,k}(r^{-N}S-I_d).
	\]
	Letting $[\sum_{j=1}^{N}M_{i,j}]_{p,q}$ denote the $p,q^{th}$ element of the $d\times d$ matrix $\sum_{j=1}^{N}M_{i,j}$, we have
	\begin{align*}
		\left\| \sum_{j=1}^{N}M_{i,j} \right\|_{\infty} 
		&= \max_{1\leq p \leq d}\sum_{q=1}^{d}[\sum_{j=1}^{N}M_{i,j}]_{p,q} \\
		&\leq \max_{1\leq p \leq d}\left( \sum_{q=1}^{d}r^{-N}[S]_{p,q} + \sum_{q=1}^{d}\sum_{\alpha=1}^{d}[\sum_{k=1}^{i-1}|C|_{i,k}]_{p,\alpha} [r^{-N}S-I_d]_{\alpha,q} \right) \\
	&=	\max_{1\leq p \leq d}\left( \sum_{q=1}^{d}r^{-N}[S]_{p,q} + \sum_{\alpha=1}^{d}[\sum_{k=1}^{i-1}|C|_{i,k}]_{p,\alpha} (r^{-N}\sum_{q=1}^{d}[S]_{\alpha,q}-1) \right) \\
	 &< \max_{1\leq p \leq d}\bigl( r^{-N}\sum_{q=1}^{d}[S]_{p,q} \bigr) < 1.
	\end{align*}
	With the last two inequalities holding as $r^{-N}\sum_{q=1}^{d}[S]_{\alpha,q}<1$ for all $\alpha\in\{1,\cdots,d\}$.
	Thus $\|M \|_{\infty} = \max_{1\leq i\leq N}( \| \sum_{j=1}^{N}M_{i,j} \|_{\infty} )<1$ and
	 \eqref{eq:rownorm} is satisfied.
\end{proof}

\section{Main Results}\label{sect:res}
 We next show that the solution $Z$ of equation \eqref{equ:resres} is in $\mathcal{WP}(r,N)$ with weight function $\phi$, when the kernel $U$ lies in $\mathcal{WP}(r,N)$ with weight function $\phi$. Once the behaviour of $Z$ is known, a variation of constants formula readily enables us to determine the asymptotic behaviour of the solution of \eqref{equ:convol}. 
 Firstly we give a lemma concerning the summability of $Z$.
\begin{lemma}\label{lm:sumZ}
	Let $Z$ be the solution of \eqref{equ:resres}. If \eqref{equ:c3} holds then $\sum_{i=0}^{N-1}\sum_{n=0}^{\infty}r^{-N(n+1)}|Z(Nn+i)|$ is finite and the following inequality holds:
	\begin{align*}
		\sum_{i=0}^{N-1}\sum_{n=0}^{\infty}r^{-N(n+1)}|Z(Nn+i)| 
		&\leq r^{-N}I +\\
		& \left(	\sum_{i=0}^{N-1}\sum_{n=0}^{\infty}r^{-N(n+1)}|U(Nn+i)|\right)
		\left(	\sum_{i=0}^{N-1}\sum_{n=0}^{\infty}r^{-N(n+1)}|Z(Nn+i)|\right).
	\end{align*}
\end{lemma}
\begin{theorem}\label{thm:res}
	Let $\{Z(n), n\in \N\}$ be the sequence which satisfies \eqref{equ:resres}. Suppose that $U\in\mathcal{WP}(r,N)$ with 
	weight function $\phi\in \mathcal{W}(r)$ such that there exists a sequence of $d\times d$ matrices $\{A_{i}\}_{i=0}^{N-1}$ and
\begin{equation}\label{equ:c2}
	  \lim_{n\to\infty}\frac{1}{\phi(Nn)}U(Nn+i)= A_i ,\quad i\in\{0,1,2,...,N-1\},   
\end{equation}
\begin{equation}\label{equ:c3}
	\max_{1\leq p\leq d} \left(\sum_{q=1}^{d} \sum_{i=0}^{N-1}\sum_{l=0}^{\infty}r^{-N(l+1)}|U(Nl+i)|_{p,q} \right) < 1 , \quad r \leq 1,
\end{equation}
for some $N\in\N$. Then $Z\in \mathcal{WP}(r,N)$ and there exists a $\{\rho_i \} \in\R^{d\times d}$ such that
\begin{equation}\label{eq:reslimit}
	\lim_{n\to\infty}\frac{1}{\phi(Nn)}Z(Nn+i) =: \rho_i.
\end{equation}
\end{theorem}

\begin{remark}
Condition \eqref{equ:c2} gives us the rate of decay of the components of $U(Nn+i)$ for each $i$. Hence it encapsulates both the decay and periodic components of the kernel. Condition \eqref{equ:c3} is imposed in order to ensure stability of the problem. While the $||\cdot||_{\infty}$ is employed here for simplicity and to ease the calculations involved, we speculate that other norms may also be possible while noting the equivalence of norms for scalar functions.  The result \eqref{eq:reslimit} is analogous to \eqref{equ:c2}, that is that the solution of \eqref{equ:resres} inherits the same rate of decay as $U$, and also retains a similar periodic component. We note that while it is possible to calculate an explicit formula for $\rho_i$, it is in general far more complicated than the constant matrix $A_i$. That such limits may in general prove rather unilluminating
may be seen from the explicit example in Section~\ref{sect:ex}.
\end{remark}
\begin{remark}
	Later, we give a partial converse to Theorem~\ref{thm:res} which illustrates the sharpness of  \eqref{equ:c2}, \eqref{equ:c3}.
\end{remark}
\begin{proof}[Proof of Theorem~\ref{thm:res}]
	We first develop a system of equations from \eqref{equ:resres}, which can be put into the form of \eqref{equ:1a}. We then focus on ensuring that all the conditions of Theorem \ref{thm:3} hold.
	From \eqref{equ:resres} we can write for $i>0$,
	\begin{align*}
		Z(Nn+i) &= \sum_{j=0}^{Nn+i-1}U(j)Z(Nn+i-1-j) \\
		&= \sum_{k=0}^{n}\sum_{j=0}^{i-1}U(Nk+j)Z(Nn+i-1-Nk-j) \\
		 &\qquad + \sum_{k=0}^{n-1}\sum_{j=i}^{N-1}U(Nk+j)Z(Nn+i-1-Nk-j) \\
		&= \sum_{j=0}^{i-1}U(j)Z(Nn+i-j-1)
		 + \sum_{k=0}^{n-1}\sum_{j=0}^{i-1}U(N(k+1)+j)Z(N(n-k-1)+i-j-1) \\
		&\qquad + \sum_{j=i}^{N-1}\sum_{k=0}^{n-1}U(Nk+j)Z(N(n-k-1)+N+i-j-1) \\
		&= \sum_{j=0}^{i-1}U_{j}(0)Z_{i-j-1}(n) + \sum_{j=0}^{i-1}\sum_{k=0}^{n-1}\bar{U}_{j}(k)Z_{i-j-1}(n-1-k) \\
		 &\qquad + \sum_{j=i}^{N-1}\sum_{k=0}^{n-1}U_j(k)Z_{N+i-j-1}(n-1-k)
	\end{align*}
	where in the last line, we set $Z_i(n) := Z(Nn+i)$; $U_i(n) := U(Nn+i)$; and $\bar{U}_{i}(n) := U_{i}(n+1)$. Thus 	
	\begin{align}\label{eq:star}
	Z_i(n) = \sum_{j=0}^{i-1}U_{j}(0)Z_{i-j-1}(n)
	 + \sum_{l=0}^{i-1}\left(\bar{U}_{i-1-l}\ast Z_l\right)(n-1)
	 + \sum_{l=i}^{N-1}\left(U_{N+i-1-l}\ast Z_l\right)(n-1).
	\end{align}
	In the case when $i=0$, a similar result is obtained, but neither the second nor the third term appear in \eqref{eq:star}.
	Thus, for $i \in \{0,1,...,N-1\}$ we generate a system of equations
	\begin{equation}\label{equ:matres}
		\un{Z}(n) = B\cdot \un{Z}(n) + (J\ast\un{Z})(n-1), \quad n\geq1,
	\end{equation}
	where $\un{Z}(n)\in \mathbb{R}^{Nd\times d}$, $B\in \mathbb{R}^{Nd\times Nd}$ and $J(n)\in\mathbb{R}^{Nd\times Nd}$ where for $p,q\in\{1,2..,N\}$ we define
	\begin{equation}\label{eq:defJ}
		[\un{Z}(n)]_{p} = Z_{p-1}(n), \quad
		B_{p,q} =
		\begin{cases}
				0, & \text{if } p\leq q, \\
				U(p-q-1), & \text{if } p>q. \\
		\end{cases}, 
	\quad
		J(n)_{p,q} =
		\begin{cases}
			U_{N+p-q-1}(n), & \text{if } p\leq q, \\
			\bar{U}_{p-q-1}(n), & \text{if } p>q.
		\end{cases}
	\end{equation}
	Note that $I-B$ is in the form given in \eqref{eq:iminusb} in Lemma \ref{lemma:matinv}, so $(I-B)^{-1}$ exists. Equation \eqref{equ:matres}  simplifies to
	\begin{equation}\label{equ:matres2}
				\un{Z}(n) = (F\ast\un{Z})(n-1), \quad n\geq1,
	\end{equation}
where $F(n) := (I-B)^{-1}J(n)$. In order to satisfy the conditions of Theorem \ref{thm:3}, we need to show that, for some weight function, $\mu$, in $\mathcal{W}(s)$, $L_{\mu}F$ exists and that
	\begin{equation}\label{equ:spec1}
		\rho\left(\sum_{l=0}^{\infty}s^{-(l+1)}F(i)\right) < 1.
	\end{equation}
	We note that a natural choice of $\mu$ is $\{\Phi(n)\}_{n\geq0} := \{\phi(Nn)\}_{n\geq0}$ as $L_{\Phi}F$ is well-defined. We note by Lemma~\ref{lemma:wgt} that $\Phi$ is in $\mathcal{W}(r^N)$.	
	Observe that $L_{\Phi}F = (I-B)^{-1}\lim_{n\to\infty}J(n)/\Phi(n)$,
	and the limit exists because
	\begin{equation}\label{eq:deflimJ}
		\left[\lim_{n\to\infty}\frac{1}{\Phi(n)}J(n)\right]_{p,q} =
		\begin{cases}
			A_{N+p-q-1}, & \text{if } p\leq q, \\
			A_{p-q-1}r^N, & \text{if } p > q.
		\end{cases}
	\end{equation}
	Turning our attention to \eqref{equ:spec1}, we see what is needed is
	\begin{equation}\label{equ:spec2}
		\rho\left(\sum_{l=0}^{\infty}r^{-N(l+1)}|F(i)|\right) < 1.
	\end{equation}
	However, by \eqref{equ:linalg} we need only check $\|\sum_{i=0}^{\infty}r^{-N(i+1)}|F(i)|$ $\|_{\infty}<1$. 
	Applying Lemma~\ref{lm:spec} we see that \eqref{equ:spec2} holds. Therefore, $L_{\Phi}\un{Z}$ exists and is given by Theorem~\ref{thm:3}. Hence, by looking at the components of $\un{Z}$ we see that $Z(nN+i)/\phi(Nn) \to \rho_i$, as $n\to\infty$.
\end{proof}
\begin{proof}[Proof of Lemma~\ref{lm:sumZ}]
Define $Z_i(n)=Z(Nn+i)$, $U_{i}(n)=U(Nn+i)$ for $i\in\{0,1\ldots,N-1\}$. Then by \eqref{equ:resres}, 
$Z_0(0) =I$, and 
\begin{align*}
	 Z_i(0) &= \sum_{p=0}^{i-1}U_{i-p-1}(0)Z_p(0), \quad i\in \{1,\ldots,N-1\}, \\
	Z_0(n) &= \sum_{l=0}^{n-1}\sum_{p=0}^{N-1}U_{N-p-1}(n-l-1)Z_p(l), \quad n\geq 1, \\
	Z_i(n) &= \sum_{l=0}^{n}\sum_{p=0}^{i-1}U_{i-p-1}(n-l)Z_p(l) + \sum_{l=0}^{n-1}\sum_{p=i}^{N-1}U_{N+i-p-1}(n-l-1)Z_p(l), \quad n\geq 1, \, 1\leq i\leq N-1.
\end{align*}
Then taking absolute values across \eqref{equ:resres} and summing we have
\begin{align*}
	\sum_{i=0}^{N-1}\sum_{n=0}^{T}r^{-N(n+1)}|Z_i(n)| \leq r^{-N}|Z_0(0)| + \sum_{i=1}^{N-1} r^{-N}|Z_i(0)| 
	&+ \sum_{n=1}^{T}r^{-N(n+1)}|Z_0(n)| \\
	 &\quad + \sum_{i=1}^{N-1}\sum_{n=1}^{T}r^{-N(n+1)}|Z_i(n)|,
\end{align*}
where $T$ is a large fixed integer. Substituting the above representations for $Z$ into this equation and permuting sums yields
\begin{align*}
\lefteqn{
	\sum_{i=0}^{N-1}\sum_{n=0}^{T}r^{-N(n+1)}|Z_i(n)|}\\
	& \leq r^{-N}I + \sum_{p=0}^{N-2}\sum_{q=0}^{N-p-2}r^{-N}|U_q(0)||Z_p(0)|  + \sum_{p=0}^{N-1}\sum_{l=0}^{T-1}\sum_{n=0}^{T-l-1}r^{-N(n+l+2)}|U_{N-p-1}(n)||Z_p(l)| \\
	&\quad+ \sum_{p=0}^{N-2}\sum_{q=0}^{N-p-2}\sum_{l=1}^{T}\sum_{n=0}^{T-l}r^{-N(n+l+1)}|U_q(n)||Z_p(l)| \\
	&\qquad+ \sum_{p=0}^{N-2}\sum_{q=0}^{N-p-2}\sum_{n=1}^{T}r^{-N(n+1)}|U_q(n)||Z_p(0)| 
	+ \sum_{p=1}^{N-1}\sum_{q=N-p}^{N-1}\sum_{l=0}^{T-1}\sum_{n=0}^{T-l-1}r^{-N(n+l+2)}|U_q(n)||Z_p(l)|.
	\end{align*}
The remainder of the calculation hinges on careful splitting and recombination of these sums, and by replacing 
$T-c$ by $T$ in various upper limits of summation. Successively, we estimate according to 
\begin{align*}
\lefteqn{
	\sum_{i=0}^{N-1}\sum_{n=0}^{T}r^{-N(n+1)}|Z_i(n)|}\\
	& \leq r^{-N}I + \sum_{p=0}^{N-2}\sum_{q=0}^{N-p-2}r^{-N}|U_q(0)||Z_p(0)| 
	+ \sum_{p=0}^{N-1}\sum_{l=0}^{T}\sum_{n=0}^{T}r^{-N(n+l+2)}|U_{N-p-1}(n)||Z_p(l)| \\
	&\quad+ \sum_{p=0}^{N-2}\sum_{q=0}^{N-p-2}\sum_{l=1}^{T}\sum_{n=0}^{T}r^{-N(n+l+1)}|U_q(n)||Z_p(l)| \\
	&\qquad+ \sum_{p=0}^{N-2}\sum_{q=0}^{N-p-2}\sum_{n=1}^{T}r^{-N(n+1)}|U_q(n)||Z_p(0)| 
	+ \sum_{p=1}^{N-1}\sum_{q=N-p}^{N-1}\sum_{l=0}^{T}\sum_{n=0}^{T}r^{-N(n+l+2)}|U_q(n)||Z_p(l)| \\
	&= r^{-N}I + \sum_{p=0}^{N-2}\sum_{q=0}^{N-p-2}\sum_{n=0}^{T}r^{-N(n+1)}|U_q(n)||Z_p(0)| 
	+ \sum_{p=0}^{N-2}\sum_{q=0}^{N-p-2}\sum_{l=1}^{T}\sum_{n=0}^{T}r^{-N(n+l+1)}|U_q(n)||Z_p(l)| \\
	&\quad+ \sum_{l=0}^{T}\sum_{n=0}^{T}r^{-N(n+l+2)}|U_{N-1}(n)||Z_0(l)| 
	+ \sum_{p=1}^{N-1}\sum_{q=N-p-1}^{N-1}\sum_{l=0}^{T}\sum_{n=0}^{T}r^{-N(n+l+2)}|U_q(n)||Z_p(l)| \\
	&= r^{-N}I + \sum_{p=0}^{N-2}\sum_{q=0}^{N-p-2}\sum_{l=0}^{T}\sum_{n=0}^{T}r^{-N(n+l+1)}|U_q(n)||Z_p(l)| 
	+ \sum_{l=0}^{T}\sum_{n=0}^{T}r^{-N(n+l+2)}|U_{N-1}(n)||Z_0(l)| \\
	&\quad+ \sum_{p=1}^{N-2}\sum_{q=N-p-1}^{N-1}\sum_{l=0}^{T}\sum_{n=0}^{T}r^{-N(n+l+2)}|U_q(n)||Z_p(l)| + \sum_{q=0}^{N-1}\sum_{l=0}^{T}\sum_{n=0}^{T}r^{-N(n+l+2)}|U_q(n)||Z_{N-1}(l)| \\
	&\leq r^{-N}I + \sum_{p=0}^{N-2}\sum_{q=0}^{N-p-2}\sum_{l=0}^{T}\sum_{n=0}^{T}r^{-N(n+l+2)}|U_q(n)||Z_p(l)| 
	+ \sum_{l=0}^{T}\sum_{n=0}^{T}r^{-N(n+l+2)}|U_{N-1}(n)||Z_0(l)| \\
	&\quad+ \sum_{p=1}^{N-2}\sum_{q=N-p-1}^{N-1}\sum_{l=0}^{T}\sum_{n=0}^{T}r^{-N(n+l+2)}|U_q(n)||Z_p(l)| 
	+ \sum_{q=0}^{N-1}\sum_{l=0}^{T}\sum_{n=0}^{T}r^{-N(n+l+2)}|U_q(n)||Z_{N-1}(l)| \\
	&= r^{-N}I +\left(\sum_{q=0}^{N-1}\sum_{n=0}^{T}r^{-N(n+1)}|U_q(n)|\right)
	\left(\sum_{p=0}^{N-1}\sum_{l=0}^{T}r^{-N(n+1)}|Z_p(l)|\right),
\end{align*}
where the last inequality holds as $1\leq r^{-N}$. Therefore by \eqref{equ:c3} 
\begin{equation} \label{eq.Ttoinfty}
		\sum_{i=0}^{N-1}\sum_{n=0}^{T}r^{-N(n+1)}|Z_i(n)| 
	\leq r^{-N}I +\left(\sum_{j=0}^{N-1}\sum_{n=0}^{\infty}r^{-N(n+1)}|U_j(n)|\right)
	\left(\sum_{i=0}^{N-1}\sum_{l=0}^{T}r^{-N(n+1)}|Z_i(l)|\right).
\end{equation} 
Due to condition \eqref{equ:c3}, we have that $\left(I-\sum_{j=0}^{N-1}\sum_{n=0}^{\infty}r^{-N(n+1)}|U_j(n)|\right)^{-1}$ exists and moreover is a non-negative matrix. Hence we have
\[
	\sum_{i=0}^{N-1}\sum_{n=0}^{T}r^{-N(n+1)}|Z_i(n)| \leq \left(I-\sum_{j=0}^{N-1}\sum_{n=0}^{\infty}r^{-N(n+1)}|U_j(n)|\right)^{-1}r^{-N}.
\]
Noting that each entry in the lefthand side of the above inequality is an increasing function of $T$ and is bounded above by a term which is independent of $T$, tells us that each entry of the matrix has a finite limit as $T\to\infty$. This proves the result. The inequality in the statement of the lemma follows by letting $T\to\infty$ in \eqref{eq.Ttoinfty}.
\end{proof}
The following corollary applies Theorem~\ref{thm:res} to \eqref{equ:convol}.
\begin{corollary}\label{thm:xsol}
	Let $\{X(n):n\in \N\}$ be the solution of \eqref{equ:convol}, $\{Z(n):n\in\N\}$ the solution of \eqref{equ:resres} and $\phi \in \mathcal{W}(r)$ and \eqref{equ:c2}, \eqref{equ:c3} hold. Let $\{\rho_l\}_{l=0}^{N-1}$ be given by Theorem \ref{thm:res} and $i\in\{0,1,...,N-1\}$.
	Suppose
	$\lim_{n\to\infty} f(Nn+i)/\phi(Nn)=L_i$. Then $\lim_{n\to\infty} X(Nn+i)/\phi(Nn)$ exists and can be calculated.
\end{corollary}
\begin{remark}
Other results in the direction of Corollary~\ref{thm:xsol} are certainly possible to state in which the rate of decay of the perturbation is different to that of the kernel or where their periods differ. The proofs follow readily by the variation of constants formula
and the facts that (i) the convolution of two sequences which lie in $\mathcal{WP}(r,N)$ also lies in $\mathcal{WP}(r,N)$.
(ii) the sum of two  sequences in $\mathcal{WP}(r,N)$ is also in $\mathcal{WP}(r,N)$.
Therefore, we do not dwell on this issue but leave it instead to the reader's imagination to consider these obvious extensions.
\end{remark}

\begin{proof}[Proof of Corollary~\ref{thm:xsol}]
By Theorem \ref{thm:res} we have $\lim_{n\to\infty} Z(Nn+i)/\phi(Nn)= \rho_i$. Using \eqref{equ:resrep} and the same argument at the start of the proof of  Theorem \ref{thm:res} we can write
	\begin{gather*}
		X(Nn+i) = Z(Nn+i)X(0) + \sum_{l=0}^{i}(Z_l*F_{i-l})(n) + \sum_{l=i+1}^{N-1}(Z_l*F_{N+i-l})(n-1),
	\end{gather*}
	where $f(0):=0$, $Z_a(b):= Z(Nb+a)$ and $F_a(b) := f(Nb+a)$, $a\in\{0,1,\ldots,N-1\},b\in\Z^{+}$.
	Define $\Phi(n) = \phi(Nn)$. Using \cite[Thm:4.3]{jaigdr:2006} and $\Phi\in\mathcal{W}(r^N)$ we obtain
	\begin{align}\label{eq:perlim}
	\lim_{n\to\infty}\frac{X(Nn+i)}{\phi(Nn)} &= \rho_i X(0) + \sum_{l=0}^{i}\rho_l\sum_{j=0}^{\infty}F_{i-l}(j) r^{-Nj}
		+ \sum_{l=0}^{i}\sum_{j=0}^{\infty}Z_l(j) r^{-Nj} L_{i-l} \notag \\
	&\qquad + \sum_{l=i+1}^{N-1}\rho_l\sum_{j=0}^{\infty}F_{N+i-l}(j)r^{-N(j+1)}
		+ \sum_{l=i+1}^{N-1}\sum_{j=0}^{\infty}Z_l(j) r^{-N(j+1)} L_{N+i-l}.
	\end{align}
which completes the proof.
\end{proof}
We close this section by noting that $Z\in \mathcal{W}(r,N)$ is in some sense only possible 
if $U\in \mathcal{W}(r,N)$. This result is a consequence of Theorem~\ref{thm:res} and Corollary~\ref{thm:xsol}. 

We note that one may show, via induction, that the solution $Z$ of \eqref{equ:resres} can be expressed as $Z(n)= U(n-1) + \sum_{j=2}^{n}U^{(*j)}(n-j)$, for $n\geq2$, with $Z(1)=U(0)$, $Z(0)=I$. Furthermore this representation allows one to show that $Z$ is also a solution of the equation $W(n+1)=(W*U)(n)$, $n\geq0$, $W(0)=I$. Hence
$(U*Z)(n) = Z(n+1) = W(n+1) = (W*U)(n) = (Z*U)(n)$. 
By rewriting \eqref{equ:resres}, we get $U(n+1)=Z(n+2)-\sum_{j=1}^{n+1} U(n+1-j)Z(j)$ for $n\geq 0$. 
Putting $Y(n)=-Z(n+1)$ we see that 
\begin{equation} \label{eq.convolconverse}
U(n+1)=-Y(n+1)+\sum_{l=0}^n U(n-l)Y(l), \quad n\geq 0.
\end{equation}
We now argue that $(U*Y)=(Y*U)$. For $n\geq0$ we have
\begin{align*}
	(U*Y)(n) = -\sum_{j=0}^{n}U(n-j)Z(j+1) = -\sum_{j=0}^{n}U(n-j)(U*Z)(j)= -(U*U*Z)(n).
\end{align*}
Similarly $(Y\ast U)(n) = -(U*Z*U)(n)$. But $Z*U=U*Z$, so $(U*Y)(n)=-(U*U*Z)(n)= -(U*(Z*U))(n)=(Y*U)(n)$. Therefore \eqref{eq.convolconverse} becomes
\begin{equation} \label{eq.convolconverse2}
U(n+1)=-Y(n+1)+\sum_{l=0}^n Y(n-l)U(l), \quad n\geq 0.
\end{equation}
which is in the form of \eqref{equ:convol}. We introduce the resolvent $R$ by 
$R(n+1)=\sum_{j=0}^n Y(n-j)R(j)$ for $n\geq 0$, where $R(0)=I$.
We now give conditions under which Theorem~\ref{thm:res} can be applied. If we suppose that $Z$ obeys 
\eqref{eq:reslimit}, then for $i=0,\ldots,N-1$ we have
\begin{equation} \label{eq.WinWrper}
\lim_{n\to\infty} \frac{Y(Nn+i)}{\phi(Nn)}=-\lim_{n\to\infty}\frac{Z(Nn+i+1)}{\phi(Nn)}
=\begin{cases}
-\rho_{(i+1)}, & i=0,\ldots, N-2\\
-r^N \rho_0, & i=N-1.
\end{cases}
\end{equation}
Moreover, the condition 
\begin{equation}\label{equ:c4}
	\max_{1\leq p\leq d} \left(\sum_{q=1}^{d} \sum_{i=0}^{N-1}\sum_{l=0}^{\infty}r^{-N(l+1)}|Z(Nl+i+1)|_{p,q} \right) < 1 , \quad r \leq 1,
\end{equation}
is equivalent to 
\begin{equation*}
	\max_{1\leq p\leq d} \left(\sum_{q=1}^{d} \sum_{i=0}^{N-1}\sum_{l=0}^{\infty}r^{-N(l+1)}|Y(Nl+i)|_{p,q} \right) < 1 , \quad r \leq 1,
\end{equation*}
and by applying Theorem~\ref{thm:res} with $Y$ in the role of $U$ and $R$ in the role of $Z$, there exist $D_i\in \mathbb{R}^{d\times d}$ for $i=0,\ldots, N-1$ such that
	$\lim_{n\to\infty} R(Nn+i)/\phi(Nn) =: D_i$. 
Using this limit in conjunction with \eqref{eq.WinWrper}, we may now apply Corollary \ref{thm:xsol} to \eqref{eq.convolconverse2} to deduce that there exist $A_i\in \mathbb{R}^{d\times d}$ for $i=0,\ldots, N-1$ such that  
$\lim_{n\to\infty} U(Nn+i)/\phi(Nn) =: A_i$. 
However we would rather replace \eqref{equ:c4} with a norm condition on $U$ (see \eqref{equ:c5} below) which must be  stronger than \eqref{equ:c3}, as this would then yield a converse with conditions closer to that of Theroem~\ref{thm:res}. By virtue of the discussion above, what remains to be proved in the converse below is that \eqref{equ:c5} implies \eqref{equ:c4}. 
\begin{theorem}\label{thm:rescon}
	Let $\{Z(n), n\in \N\}$ be the sequence which satisfies \eqref{equ:resres}. Suppose that $Z\in \mathcal{WP}(r,N)$ with weight function $\phi$ in $\mathcal{W}(r)$ so that there is a sequence of $d\times d$ matrices $\{\rho_{i}\}_{i=0}^{N-1}$ and 
\begin{equation*} 
	  \lim_{n\to\infty}\frac{1}{\phi(Nn)}Z(Nn+i)= \rho_i ,\quad	 i\in\{0,1,2,...,N-1\}.  
\end{equation*}
Also suppose  
\begin{equation}\label{equ:c5}
	\max_{1\leq p\leq d} \left(\sum_{q=1}^{d} \sum_{i=0}^{N-1}\sum_{l=0}^{\infty}r^{-N(l+1)}|U(Nl+i)|_{p,q} \right) 
	< \frac{1}{1+r^{-N}},
	\quad r \leq 1,
\end{equation}
holds for some $N\in\N$. Then $U\in \mathcal{WP}(r,N)$ with weight function $\phi$ i.e., there exists $\{A_i \} \in\R^{d\times d}$ such that
\begin{equation*} 
	\lim_{n\to\infty}\frac{1}{\phi(Nn)}U(Nn+i) =A_i, \quad	 i\in\{0,1,2,...,N-1\}.
\end{equation*}
\end{theorem}
\begin{remark}
	In the special case where there is no periodicity (i.e., $N=1$) the necessary and sufficient nature of Theorems~\ref{thm:res} and \ref{thm:rescon} is an improvement on the sufficient nature of the conditions of Theorem~\ref{thm:3}.
\end{remark}
\begin{proof}
	We show that \eqref{equ:c5} inplies \eqref{equ:c4}. Regrouping the terms in \eqref{equ:c4}, one deduces
	\begin{align*}
		&\max_{1\leq p\leq d} \left(\sum_{q=1}^{d} \sum_{i=0}^{N-1}\sum_{l=0}^{\infty}r^{-N(l+1)}|Z(Nl+i+1)|_{p,q} \right) \notag\\
		&\quad = \max_{1\leq p\leq d} \sum_{q=1}^{d} \left( \sum_{j=1}^{N-1}r^{-N}|Z(j)|_{p,q} 
		+\sum_{j=1}^{N-1}\sum_{l=1}^{\infty}r^{-N(l+1)}|Z(Nl+j)|_{p,q} + \sum_{l=1}^{\infty}r^{-Nl}|Z(Nl)|_{p,q} \right). 
		\end{align*}
		Hence, using $1\leq r^{-N}$,
		\begin{align}
		&\max_{1\leq p\leq d} \left(\sum_{q=1}^{d} \sum_{i=0}^{N-1}\sum_{l=0}^{\infty}r^{-N(l+1)}|Z(Nl+i+1)|_{p,q} \right) \notag\\	
		&\quad \leq \max_{1\leq p\leq d} \sum_{q=1}^{d} \left( \sum_{j=1}^{N-1}r^{-N}|Z(j)|_{p,q} 
		+\sum_{j=1}^{N-1}\sum_{l=1}^{\infty}r^{-N(l+1)}|Z(Nl+j)|_{p,q} + \sum_{l=1}^{\infty}r^{-N(l+1)}|Z(Nl)|_{p,q} \right) \notag\\
		&\quad = \max_{1\leq p\leq d} \sum_{q=1}^{d} \left(\sum_{j=0}^{N-1}\sum_{l=0}^{\infty}r^{-N(l+1)}|Z(Nl+j)|_{p,q} 
		-r^{-N}|Z(0)|_{p,q}  \right) \notag\\
		&\quad = \max_{1\leq p\leq d} \sum_{q=1}^{d} \left(\sum_{j=0}^{N-1}\sum_{l=0}^{\infty}r^{-N(l+1)}|Z(Nl+j)|_{p,q} \right) 
		-r^{-N},  \label{eq:normZ}
	\end{align}
	with the last equality holding as $Z(0)=I$, whose rows sum to one, which is independent of $p$. Define the matrices
	$A=\sum_{i=0}^{N-1}\sum_{n=0}^{\infty}r^{-N(n+1)}|Z(Nn+i)|$ and $B=\sum_{i=0}^{N-1}\sum_{n=0}^{\infty}r^{-N(n+1)}|U(Nn+i)|$. Then Lemma~\ref{lm:sumZ} gives
 $A\leq r^{-N}I + BA$ or equivalently
	$A\leq(I-B)^{-1}r^{-N}$, with the direction of the inequality being preserved due to $B\geq0$ and the expression $(I-B)^{-1}=\sum_{l=0}^{\infty}B^l$, which is valid due to \eqref{equ:c5}. Taking the infinity norm on both sides of this inequality gives
	\[
		\left\| A  \right\|_{\infty} \leq \, \left\|\sum_{l=0}^{\infty}B^l\right\|_{\infty} r^{-N} 
		\leq  r^{-N}\sum_{l=0}^{\infty}\left\| B^l\right\|_{\infty} 
		\leq  r^{-N}\sum_{l=0}^{\infty}\| B\|_{\infty}^l = r^{-N}\frac{1}{1-\|B\|_{\infty}}.
	\]
	Combining this with \eqref{eq:normZ} gives
	\[
		\max_{1\leq p\leq d} \left(\sum_{q=1}^{d} \sum_{i=0}^{N-1}\sum_{l=0}^{\infty}r^{-N(l+1)}|Z(Nl+i+1)|_{p,q} \right)
		\leq \,	\| A  \|_{\infty} -r^{-N} \leq r^{-N}\frac{1}{1-\| B\|_{\infty}} -r^{-N}.
	\]
	Thus if  	
		$r^{-N}/(1-\| B\|_{\infty}) -r^{-N}<1$ 
	we have our result. But this inequality is equivalent to 
		$\|B\|_{\infty} < 1/(1+r^{-N})\leq1/2<1$, 	which is true by hypothesis.
\end{proof}

\section{Examples}\label{sect:ex}
	We provide an application of the above theory to analysing the memory characteristics of auto--regressive conditional heteroskedastic processes of order infinity. We briefly give some background details pertaining to the memory properties of ARCH($\infty$) processes, see \cite{Engle:1995,lgpkrl:2000,pkrl:2000,Zaffaroni:2004} for more detail.
	\begin{definition}  \label{def:arch}
	A random sequence $\{X(k),k\in\Z\}$ is said to satisfy ARCH($\infty$) equations if there is a sequence of independent and identically distributed (i.i.d.) nonnegative random variables $\{\xi(k),k\in \Z\}$ such that
		\begin{gather} \label{eq:1a} \tag{ARCH}
			X(k) = \biggl(a + \sum_{j=1}^{\infty}b(j)X(k-j)\biggr)\xi(k),
		\end{gather}
	where $a\geq0$, $b(j)\geq0$, for $j = \{1,2,...\}$.
	\end{definition}
The condition
	\begin{equation}\label{eq:con1} 
		\mathbb{E}[\xi(0)]\sum_{j=1}^{\infty}b(j)<1,
	\end{equation}
	is imposed in \cite{pkrl:2000} to show the presence of a strictly stationary solution of \eqref{eq:1a}. While the condition
\begin{equation}\label{eq:con2} 
	\mathbb{E}[\xi(0)^2]^{\frac{1}{2}}\sum_{j=1}^{\infty}b(j)<1.
\end{equation}
is shown in \cite{lgpkrl:2000} to imply a unique weakly stationary solution in the class of all stationary solutions with finite second moment, it is further shown in \cite{lgpkrl:2000} that \eqref{eq:con2} implies the positivity and absolute summability of the autocovariance function of stationary solutions of \eqref{eq:1a} (ie. long memory is ruled out).	
	
	Moreover \cite{lgpkrl:2000} establishes a moving average representation for \eqref{eq:1a}. It is remarked in \cite[pp.16]{lgpkrl:2000} and \cite[pp.154]{Zaffaroni:2004} that it is the asymptotic behaviour of the coefficients in this moving average representation which impart the rate of decay of the auto--covariance function of \eqref{eq:1a}. The precise influence of these coefficients is the subject of a result in \cite{Zaffaroni:2004}. We give the set up of this theorem; let $\psi(L) = 1 - \mathbb{E}[\xi(0)]\sum_{j=1}^{\infty}b(j)L^j$, where $L$ is the lag operator (i.e. $L(X(k))=X(k-1)$) and define $\nu(k) :=X(k) - \mathbb{E}[\xi(0)]\sigma(k)$, where $\nu$ is a martingale difference sequence, ie. $\mathbb{E}[\nu(k)|\mathcal{F}_{k-1}]=0$ and $\mathcal{F}_{k-1}$ is the $\sigma$-algebra generated by $\xi_{k-1},\xi_{k-2},...$. Then from \eqref{eq:1a} we have
$\psi(L)X(k) = a\mathbb{E}[\xi(0)] + \nu(k)$. Assuming the invertibility condition 
(\cite{Zaffaroni:2004}, \cite[Lm.4.1]{lgpkrl:2000}),
\begin{equation}\label{eq:delta}
 \exists\text{ } D(z) = \sum_{j=0}^{\infty}\delta(j)z^j = \frac{1}{\psi(z)}, \quad \delta(0)=1, 
 \quad \text{ for all } |z|\leq1 \quad \text{such that } \sum_{j=0}^{\infty}\delta^2(j)<\infty,
\end{equation}
then
\begin{equation}
	X(k) = a\mathbb{E}[\xi(0)]\sum_{j=0}^{\infty}\delta(j) + \sum_{j=0}^{\infty}\delta(j)\nu(k-j).
\end{equation}
	Conditions for weak stationary are  examined in \cite{Zaffaroni:2004}. In particular the following condition
	is weaker than \eqref{eq:con2}:
\begin{equation}\label{eq:con3} 
	\mathbb{E}[(\xi(0)-\mathbb{E}[\xi(0)])^2]\sum_{u=-\infty}^{\infty}\chi_{\delta}(u)\chi_{b^*}(u)<1.
\end{equation}
Here $\chi_{c}(u) := \sum_{j=0}^{\infty}c(j)c(j+|u|)$ for any sequence $c\in\ell^2(\Z^{+})$, $b^*(j)=0$ if $j=0$, and $b^*(j) = b(j)$ otherwise. \eqref{eq:con3} implies absolute summability of the auto--covariance function, so both 
\eqref{eq:con3} and \eqref{eq:con2} rule out long memory. 
	In particular \cite[Thm.1]{Zaffaroni:2004} shows that the auto--covariance function of \eqref{eq:1a} obeys $\text{Cov}[X(k),X(k+u)]=C\chi_{\delta}(u)$ for some $0<C<\infty$, where $\chi_{\delta}(u)=\sum_{j=0}^{\infty}\delta(j)\delta(j+|u|)$, $u\in\{0,\pm1,...\}$. Regarding \cite[Thm.2]{Zaffaroni:2004}, we demonstrate some flaws concerning the asymptotic decay of the autocovariance function following that of $\delta$.

	\cite[Thm.2]{Zaffaroni:2004} asserts that if there exists a function $\delta$, defined according to \cite[pp.149]{Zaffaroni:2004}, and \eqref{eq:con1} and
\begin {equation} \label{eq:slower}
	\lim_{k\to\infty}\frac{b(k)}{\zeta^k} =\infty, \quad \text{ for any } 0<\zeta<1,
\end{equation}
hold, then
\begin{align} \label{eq:dasp}
	\chi_{\delta}(k) \sim C_1 b(k), \quad (k\to\infty),
\end{align}
for some $0<C_1<\infty$, with $c(x)\sim d(x)$ as $x\to x_0$, meaning that $c(x)/d(x)\to 1$.

In the forth--coming paper, \cite{jajd}, it is shown that $\delta$ satisfies the following equation,
\begin{equation}\label{eq:deltadiff}
	\delta(n) =
		\mathbb{E}[\xi(0)]\sum_{j=0}^{n-1}b(n-j)\delta(j),	 \quad n\geq 1, \quad \delta(0)=1.
\end{equation}
Indeed one can think of $\delta$ as a resolvent for a Volterra equation, derived in \cite{jajd}, which is satisfied by the auto--covariance function of the ARCH($\infty$) process.

We consider the sufficiently simple case of a scalar Volterra equation where the kernel has a `two--periodic' ($N=2$) component. We believe that this example is instructive in demonstrating the complexity of the calculations for higher $d$ or $N$, while retaining results which are eminently verifiable.

The idea of the example is that if $b$ obeys \eqref{eq:slower} and also contains a periodic component then $\chi_{\delta}$ will have a similar rate of decay to $b$ but their periodic components will not be in phase and hence $b\not\sim \chi_\delta$.
 Our first illustration of the theory deals with the ratio of $\delta/\phi$; the second uses this result to analyse $\chi_{\delta}/\phi$.
\begin{example}\label{ex:delsimb}
\end{example}
	We can take $\l:=\mathbb{E}[\xi(0)]>0$ as if $\l=0$ then $\xi(n)=0$ for all $n\in\Z^+$. Let $\l b(2n+i+1)/\phi(2n)\to a_i>0$ for $i\in\{0,1\}$, for some $\phi\in\mathcal{W}(1)$ and $a_0\not=a_1$. Let \eqref{eq:con1} hold. Observing that \eqref{eq:deltadiff} is of the form of \eqref{equ:resres}, we apply Theroem~\ref{thm:res} to \eqref{eq:deltadiff} giving,
\begin{align*}
	d_0 := \lim_{n\to\infty}\frac{\d(2n)}{\phi(2n)} = a_0T_0 + a_1T_1, \qquad
	d_1 := \lim_{n\to\infty}\frac{\d(2n+1)}{\phi(2n)} = a_1T_0 + a_0T_1,
\end{align*}
where
	$T_0 = \Lambda(2S_0(1-S_1)),$	$T_1 = \Lambda(S_0^2 + (1-S_1)^2)$,
	$\Lambda = \bigl((1-S_1)^2-S_0^2\bigr)^{-2}$ and $S_i = \l\sum_{j=0}^{\infty}b(2j+i+1)$.
\begin{remark}
	In order to achieve $\delta\sim\phi$ (or $d_0=d_1$) one might consider $T_0 = T_1$, this however leads to $S_0 + S_1 = 1$, ie. a contradiction of \eqref{eq:con1}. Hence in general $\delta$ is not asymptotic to $\phi$. 
\end{remark}

\begin{remark}\label{rk:delnum}	
We provide a numerical illustration where all of the limits in Example~\ref{ex:delsimb} may be computed explicitly.
Define $\phi(n)=n^{-2}$ for all $n\geq1$ and $\phi(0)=2$. Let $b(j) =a_1j^{-2}$ for $j/2\in\N$, $b(j) =a_0j^{-2}$ for $j/2\not\in\N$, where $a_0:=0.5$ and $a_1:=0.25$. Furthermore let $\{\xi(n)\}_{n\in\Z}$ be an i.i.d. non--negative stochastic process with mean equal to unity (ie. $\l=1$). Thus it is calculated that
\begin{align*} 
	S_0 = a_0\l\sum_{j=0}^{\infty}\frac{1}{(2j+1)^2} = \frac{\pi^2}{16}, \quad
	S_1 = a_1\l\sum_{j=0}^{\infty}\frac{1}{2^{2}(j+1)^{2}} = \frac{\pi^2}{96}.
\end{align*}
Noting that $S_0+S_1<1$, one can evaluate $\Lambda, T_0$ and $T_1$ respectively and hence $d_0$ and $d_1$. Indeed
$\Lambda = 5.55073...$, $T_0 = 6.14391...$ and $T_1 = 6.58015...$, which gives $d_0 = 4.71699...$ and $d_1 = 4.82605....$
\end{remark}

\begin{example}\label{ex:acfsimb}
\end{example}
	 We show that while it is possible to have \eqref{eq:slower} one need not have \eqref{eq:dasp}. 
	 We proceed with the same set up as in Example~\ref{ex:delsimb}, noting that \eqref{eq:slower} is satisfied.
	 Let $\phi$ be asymptotic to a decreasing sequence. Now observe,
\begin{align*}
	\chi_{\delta}(2u) &= \sum_{j=0}^{\infty}\delta(2(j+u))\delta(2j) + \sum_{j=0}^{\infty}\delta(2(j+u)+1)\delta(2j+1), \\
	\chi_{\delta}(2u+1) &= \sum_{j=0}^{\infty}\delta(2(j+u)+1)\delta(2j) + \sum_{j=0}^{\infty}\delta(2(j+u+1))\delta(2j+1).
\end{align*}
Thus for some sufficiently large positive integer $M$, we have
\begin{align*}
	\frac{\chi_{\delta}(2u)}{\phi(2u)} &= \sum_{j=0}^{M}\frac{\delta(2(j+u))}{\phi(2u)}\delta(2j)
	+ \sum_{j=0}^{M}\frac{\delta(2(j+u)+1)}{\phi(2u)}\delta(2j+1) \\
	&\quad + \sum_{j=M+1}^{\infty}\frac{\delta(2(j+u))}{\phi(2u)}\delta(2j)
	+ \sum_{j=M+1}^{\infty}\frac{\delta(2(j+u)+1)}{\phi(2u)}\delta(2j+1).
\end{align*}
For the third sum, recalling that $\delta\in\ell^1(\Z^+)$ as \eqref{eq:con1} holds,
\begin{align*}
	 \sum_{j=M+1}^{\infty}\frac{\delta(2(j+u))}{\phi(2u)}\delta(2j) &= \sum_{j=M+1}^{\infty}\frac{\delta(2(j+u))}{\phi(2(j+u))}\frac{\phi(2(j+u))}{\phi(2u)}\delta(2j)
	 \leq 4\,d_0\sum_{j=M+1}^{\infty}\delta(2j).
\end{align*}
The fourth sum can be treated similarly. Recalling the non-negativity of $\delta$, we have
\[
	\lim_{M\to\infty}\lim_{u\to\infty}\sum_{j=M+1}^{\infty}\frac{\delta(2(j+u))}{\phi(2u)}\delta(2j)
	=\lim_{M\to\infty}\lim_{u\to\infty}\sum_{j=M+1}^{\infty}\frac{\delta(2(j+u)+1)}{\phi(2u)}\delta(2j+1) =0.
\]
For the first sum we see
\[	\lim_{M\to\infty}\lim_{u\to\infty}\sum_{j=0}^{M}\frac{\delta(2(j+u))}{\phi(2u)}\delta(2j)
	= \lim_{M\to\infty}d_0\sum_{j=0}^{M}\delta(2j) = d_0\sum_{j=0}^{\infty}\delta(2j),
\]
and a similar calculation applies to the second sum. Thus, after a similar analysis of $\chi_{\delta}(2u+1)$ we have
\begin{align*}
	\lim_{u\to\infty}\frac{\chi_{\delta}(2u)}{\phi(2u)} &= d_0\sum_{j=0}^{\infty}\delta(2j) + d_1\sum_{j=0}^{\infty}\delta(2j+1) = a_0\tau_0+a_1\tau_1, \\
	\lim_{u\to\infty}\frac{\chi_{\delta}(2u+1)}{\phi(2u)} &= d_1\sum_{j=0}^{\infty}\delta(2j) + d_0\sum_{j=0}^{\infty}\delta(2j+1) = a_0\tau_1+a_1\tau_0,
\end{align*}
where
\[
	\tau_0 = T_0\sum_{j=0}^{\infty}\delta(2j) + T_1\sum_{j=0}^{\infty}\delta(2j+1), \quad
	\tau_1 = T_1\sum_{j=0}^{\infty}\delta(2j) + T_0\sum_{j=0}^{\infty}\delta(2j+1).
\]
Thus for $\chi_{\delta}\sim b$ we need
$\lim_{u\to\infty} \chi_{\delta}(2u)/b(2u) =\lim_{u\to\infty} \chi_{\delta}(2u+1)/b(2u+1)$,
which is equivalent to $\tau_0(a_0-a_1)(a_0+a_1)/(a_0a_1)=0$, 
which can only occur if either $a_0=a_1$ or $\tau_0=0$. The first is ruled out by hypothesis. For the second, summing over \eqref{eq:deltadiff} for both $\delta(2n)$ and $\delta(2n+1)$ gives
\[
	\sum_{j=0}^{\infty}\delta(2j)= \frac{(1-S_1)}{(1-S_1)^2-S_0^2}, \quad \sum_{j=0}^{\infty}\delta(2j+1)=\frac{S_0}{(1-S_1)^2-S_0^2},
\]
which gives $\tau_0$ the representation
\[
	\tau_0 = \frac{\Lambda S_0(S_0^2+3(1-S_1)^2)}{(1-S_1)^2-S_0^2}.
\]
Thus, $\tau_0$ cannot be equal to zero (as otherwise $a_0=0$). Thus, while $b(i)/\zeta^i\to\infty$ as $i\to\infty$ for any $0<\zeta<1$ we do not have $\chi_{\delta}(u)\sim Cb(u)$, as $u\to\infty$, for some $0<C<\infty$.

\begin{remark} 
Following on from Remark~\ref{rk:delnum} one can compute the various limits and infinte sums in Example~\ref{ex:acfsimb}, ie. $\sum_{j=0}^{\infty}\delta(2j),\sum_{j=0}^{\infty}\delta(2j+1),\tau_0$ and $\tau_1$ respectively and hence we have
\begin{align*}
	\lim_{u\to\infty}\frac{\chi_{\delta}(2u)}{b(2u)} = \l(\frac{a_0}{a_1}\tau_0+\tau_1) = 67.9375\ldots, \qquad 
	\lim_{u\to\infty}\frac{\chi_{\delta}(2u+1)}{b(2u+1)} = \l(\frac{a_1}{a_0}\tau_0+\tau_1) = 34.1128\ldots
\end{align*}
Thus as both $\Lambda$ and $\tau_0$ are positive (approximately 5.55073 and 22.5498 respectively), we have that the above two limts are unequal and hence $\chi_{\delta}(u)\not\sim Cb(u)$ as $u\to\infty$ for some $0<C<\infty$. 
\end{remark}

\providecommand{\bysame}{\leavevmode\hbox
to3em{\hrulefill}\thinspace}
\providecommand{\MR}{\relax\ifhmode\unskip\space\fi MR }
\providecommand{\MRhref}[2]{%
  \href{http://www.ams.org/mathscinet-getitem?mr=#1}{#2}
} \providecommand{\href}[2]{#2}


\begin{thebibliography}{10}

\bibitem{jajd}
J.~A.~D.~Appleby and J.~Daniels, \emph{On the decay rates of the autocovariance function of Stationary ARCH($\infty$) Models}, Unpublished results. 

\bibitem
{jaigdr:2006}
J.~A.~D. Appleby, I.~Gy{\H o}ri, and D.~W. Reynolds, \emph{On exact convergence rates for solutions of linear systems of Volterra difference equations}, J. Difference Equations and Applications, {\bf 12} (2006), 1257-1275.


\bibitem{sesm:1998}
S.~Elaydi and S.~Murakami, \emph{Uniformly asymptotic stability in linear Volterra difference equations}, J. Difference Equations and Applications, {\bf 3} (1998) 203-218.

\bibitem{Engle:1995}
R.~F.~Engle, ``ARCH Selected Readings, Advanced Texts in Econometrics", Oxford University Press Inc., New York, 1995.

\bibitem{lgpkrl:2000}
L.~Giraitis, P.~Kokoszka and R.~Leipus, \emph{Stationary ARCH Models:
 Dependence Structure and Central Limit Theorem}, Economic Theory, {\bf 16} (2000), 3-22.

\bibitem{MR2722631}
 I.~Gy{\H{o}}ri and D.~W.~Reynolds, \emph{On asymptotically periodic solutions of linear discrete Volterra equations}, Fasc. Math., {\bf 44} (2010), 53--67.


\bibitem{pkrl:2000}
P.~Kokoszka and R.~Leipus, \emph{Change Point Estimation in ARCH models}, Bernoulli \textbf{6} (2000), 513-539.

\bibitem{songbaker:2003}
Y.~Song and C.~T.~H.~Baker, \emph{Perturbation Theory for Discrete Volterra Equations}, J. Difference Equations and Applications, {\bf 9} (2003), 969-987.

\bibitem{songbaker:2004}
Y.~Song and C.~T.~H.~Baker, \emph{Perturbation of Volterra Difference Equations}, J. Difference Equations and Applications, {\bf 10} (2004) 379-397.

\bibitem{Zaffaroni:2004}
P.~Zaffaroni, \emph{Stationarity and Memory of ARCH($\infty$) Models}, Economic Theory, {\bf 20} (2004), 147-160.

\end{thebibliography}

\end{document}